\documentclass{amsart}
\usepackage[alphabetic]{amsrefs}
\usepackage{graphicx}
\usepackage{xcolor}

\newtheorem{theorem}{Theorem}[section]
\newtheorem{lemma}[theorem]{Lemma}
\newtheorem{proposition}[theorem]{Proposition}
\newtheorem{corollary}[theorem]{Corollary}
\theoremstyle{remark}
\newtheorem{remark}[theorem]{Remark}

\newcommand{\N}{\mathbb{N}}
\newcommand{\R}{\mathbb{R}}
\newcommand{\cR}{\mathcal{R}}
\renewcommand{\epsilon}{\varepsilon} 
\renewcommand{\phi}{\varphi} 
\renewcommand{\Re}{\operatorname{Re}}
\newcommand{\loc}{\mathrm{loc}}

\numberwithin{equation}{section}

\usepackage{microtype}

\mathchardef\ordinarycolon\mathcode`\:
\mathcode`\:=\string"8000
\begingroup \catcode`\:=\active
\gdef:{\mathrel{\mathop\ordinarycolon}}
\endgroup

\title[Time derivative in parabolic Signorini problem]{Boundedness and continuity of the time derivative in the parabolic Signorini problem}
\author{Arshak Petrosyan}
\address[AP]{Department of Mathematics, Purdue University, West Lafayette, IN 47907, USA}
\email{arshak@purdue.edu}
\author{Andrew Zeller}
\address[AZ]{Department of Mathematics, Purdue University, West Lafayette, IN 47907, USA}
\email{zellera@purdue.edu}


\begin{document}

\begin{abstract} We prove the boundedness of the time derivative in
  the parabolic Signorini problem, as well as establish its H\"older
  continuity at regular free boundary points.
\end{abstract}

\maketitle

\section{Introduction and main results}
\label{sec:introduction}

Let $v$ be a weak solution of the \emph{parabolic Signorini problem}
\begin{align}
  \label{eq:sig1}
  \Delta v-\partial_t v =0&\quad\text{in }Q_1^+:=B_1^+\times(-1,0],\\
  \label{eq:sig2}
  v\geq \phi,\quad -\partial_{x_n} v\geq 0,\quad (v-\phi)\partial_{x_n}
  v=0&\quad\text{on }Q_1':=B_1'\times(-1,0],\\
  \label{eq:sig3}
  v(\cdot,-1)=\phi_0&\quad\text{in }B_1,
\end{align}
to be understood in the appropriate integral sense, where
$\phi:Q_1'\to \R$ is the \emph{thin obstacle} and $\phi_0$ is the
initial data satisfying the compatibility condition
$\phi_0\geq \phi(\cdot,0)$ on $B_1'$.  This kind of unilateral problem
appears in many applications, such as thermics (boundary heat
control), biochemistry (semipermeable membranes and osmosis), and
elastostatics (the original Signorini problem). It also serves as a
prototypical example of parabolic variational inequalities. We refer
to the book \cite{DL} for the derivation of such models as well as for
some basic existence and uniqueness results, and to \cite{DGPT} for
more recent results on the problem.

One of the main objects of study in the parabolic Signorini problem is
the apriori unknown \emph{free boundary}
$$
\Gamma(v):=\partial_{Q_1'}(\{v>\phi\}\cap Q_1'),
$$
which separates the regions where $v=\phi$ and $\partial_{x_n}v=0$
(here $\partial_{Q_1'}$ denotes the boundary in the relative topology
of $Q_1'$).

It is known that if $\phi$ is sufficiently regular, namely
$\phi\in H^{2,1}(Q_1')$ (see the end of the introduction for the
notations) then the Lipschitz regularity of $\phi_0$ in
$B_1^+\cup B_1'$ implies the local boundedness of the spatial gradient
$\nabla v$ in $Q_1^+\cup Q_1'$ (see \cite[Lemma~6]{AU1}), which then
implies the H\"older continuity
$\nabla v\in H^{\gamma,\gamma/2}_{\loc}(Q_1^+\cup Q_1')$ (see
\cite[Theorem~2.1]{AU2}), for some $\gamma>0$. Recently, it was shown
in \cite{DGPT} that $v\in H^{3/2,3/4}_{\loc}(Q_1^+\cup Q_1')$, which
is the optimal regularity of $v$, at least in the space variables $x$.
The paper \cite{DGPT} also gives a comprehensive treatment of the
problem from the free boundary regularity point of view, based on
Almgren-, Monneau-, and Weiss-type monotonicity formulas.

The aim of this paper is to obtain a better regularity in the time
variable $t$ for the solutions of the parabolic Signorini problem
above and to complement the results of \cite{DGPT}. It is known
already from \cite[Lemma~7]{AU2} that if the initial data
$\phi_0\in W^{2}_{\infty}(B_1^+)$, then the time derivatives
$\partial_t v$ will also be locally bounded in $Q_1^+\cup Q_1'$. This
assumption on the initial data $\phi_0$, however, is rather
restrictive and excludes a ``standard'' time-independent solution (for
$\phi\equiv 0$)
$$
v(x,t)=\Re(x_{n-1}+ix_n)^{3/2},\quad x_n\geq 0,
$$
which is clearly not in $W^2_{\infty}$.

Our first result shows that $\partial_t v$ is in fact bounded, without
any extra assumptions on the initial data, even though we will require
a bit more regularity on the thin obstacle $\phi$.

\begin{theorem}\label{thm:vt-bdd} Let $v\in
  H^{3/2,3/4}(Q_1^+\cup Q_1')$
  be a solution of the Signorini problem
  \textup{\eqref{eq:sig1}--\eqref{eq:sig2}} with
  $\phi\in H^{4,2}(Q_1')$. Then $\partial_t v$ is locally bounded in
  $Q_1^+\cup Q_1'$ and moreover
$$
\|\partial_t v\|_{L_\infty(Q_{1/2}^+)}\leq
C_n\left(\|v\|_{L_2(Q_1^+)}+\|\phi\|_{H^{4,2}(Q_1')}\right).
$$
\end{theorem}
We prove this theorem in \S\ref{sec:bound-time-deriv}. In fact,
instead of asking $\phi\in H^{4,2}(Q_1')$ it is sufficient to assume
that $\partial_t(\Delta_{x'}\phi-\partial_t\phi)\in L_\infty(Q_1')$.

\medskip Our second result is that $\partial_t v$ is continuous at
so-called regular free boundary points (see \S\ref{sec:hold-cont-time}
for the definition).

\begin{theorem}\label{thm:vt-cont-reg} Let $v$ be as in
  Theorem~\ref{thm:vt-bdd}. Then $\partial_t v$ continuously equals to
  $\partial_t \phi$ at regular free boundary points.
\end{theorem}

In fact, in \S\ref{sec:hold-cont-time} we prove a more precise version
of this theorem (Theorem~\ref{thm:ut-Holder-local}), which shows the
H\"older continuity of $\partial_t v$ at regular points.

At the end of the paper we state a direct corollary on the higher
regularity of the free boundary in the $t$ variable near regular
points (see Corollary~\ref{cor:C1a-xt}). When the thin obstacle
$\phi\equiv 0$, Theorem~\ref{thm:vt-cont-reg} can be used to make an
iterative step in the application of a higher-order boundary Harnack
principle for parabolic slit domains and establish the $C^\infty$
regularity (both in $x$ and $t$) of the free boundary near regular
points (see \cite{BSZ}).

\subsection*{Notation} Throughout the paper we use the following
conventions and notations.
\begin{itemize}
\item $\R^n$ stands for the $n$-dimensional Euclidean space. For
  $x=(x_1,\ldots,x_n)\in\R^n$, we typically denote
  $x'=(x_1,\ldots,x_{n-1})$ and $x''=(x_1,\ldots,x_{n-2})$. We also
  routinely identify $x'\in\R^{n-1}$ with
  $(x',0)\in\R^{n-1}\times\{0\}\subset\R^n$.
\item $B_r(x_0)$, $B_r'(x_0)$, $B_r''(x_0)$ stand for balls of radius
  $r>0$ centered at $x_0$ in $\R^n$, $\R^{n-1}$, $\R^{n-2}$,
  respectively. We drop the center from the notation if $x_0=0$. We
  also denote $B_r^\pm(x_0)=B_r(x_0)\cap\{\pm x_n>0\}$.
\item $Q_r(x_0,t_0)=B_r(x_0)\times(t_0-r^2,t_0]$ is the parabolic
  cylinder, with similar definitions for $Q'_r$, $Q''_r$, $Q^\pm_r$.

\item For parabolic functional spaces, we use notations similar to
  those in \cite{LSU} and \cite[\S2.2]{DGPT}. In particular,
  $H^{\ell,\ell/2}(E)$ for $\ell=m+\gamma$, $m\in\N\cup\{0\}$,
  $\gamma\in (0,1]$ is the space of functions such that the partial
  derivatives $\partial_x^\alpha\partial_t^j u$ are $\gamma$-H\"older
  in $x$ and $\gamma/2$-H\"older in $t$ for the derivatives of the
  parabolic order $|\alpha|+2j\leq m$ and $(1+\gamma)/2$-H\"older in
  $t$ if $|\alpha|+2j\leq m-1$. $L_p(E)$ stands for the Lebesgue
  space, and $W^{2m,m}_p(E)$ is the Sobolev space of functions such
  that $\partial_x^\alpha\partial_t^j u\in L_p(E)$ for
  $|\alpha|+2j\leq 2m$.
\end{itemize}

\section{Boundedness of the time derivative}
\label{sec:bound-time-deriv}

We first reduce the problem to the case of zero thin obstacle, at the
expense of getting nonzero right hand side in the governing
equation. Namely, let
$$
u(x,t):=v(x,t)-\phi(x',t).
$$
Then we have
\begin{align}
  \label{eq:sig-u1}
  \Delta u-\partial_t u =f:=\partial_t\phi-\Delta_{x'}\phi&\quad\text{in }Q_1^+,\\
  \label{eq:sig-u2}
  u\geq 0,\quad -\partial_{x_n} u\geq 0,\quad u\partial_{x_n} u=0&\quad\text{on }Q_1'.
\end{align}
It will also be convenient to extend the function $u$ by the even
symmetry in the $x_n$ variable to the entire cylinder $Q_1$:
$$
u(x',-x_n,t)=u(x',x_n,t).
$$
Then the extended function will satisfy
$$
\Delta u-\partial_t
u=f+2(\partial_{x_n}^+u)\mathcal{H}^{n}\big|_{\Lambda(u)}\quad\text{in
}Q_1,
$$
in the sense of distributions, where $f$ is also extended by the even
symmetry in $x_n$ to all of $Q_1$,
$\partial_{x_n}^+u(x',0,t)=\partial_{x_n}u(x',0+,t)$ for
$(x',t)\in Q_1'$, $\mathcal{H}^n$ is the $n$-dimensional Hausdorff
measure, and
\begin{align*}
  \Lambda(u):=&\{(x',t)\in Q_1':u(x',0,t)=0\}\\
  =&\{(x',t)\in Q_1':v(x',0,t)=\phi(x',t)\}
\end{align*}
is the so-called \emph{coincidence set}.

\begin{proof}[Proof of Theorem~\ref{thm:vt-bdd}] For $u$ solving
  \eqref{eq:sig-u1}--\eqref{eq:sig-u2} and a small $h>0$ consider the
  incremental quotient in the time variable
$$
U_h(x,t)=\frac{u(x,t)-u(x,t-h)}{h},\quad (x,t)\in
Q_{3/4}.
$$
Let us also denote
$$
F_h(x,t)=\frac{f(x,t)-f(x,t-h)}{h},\quad (x,t)\in
Q_{3/4}.
$$
Note that $U_h\in H^{3/2,3/4}(Q_{3/4}^\pm\cup Q_{3/4}')$ and
$F_h\in H^{2,1}(Q_{3/4})$, from the assumption that the thin obstacle
$\phi\in H^{4,2}(Q_1')$.

We then have the following key observation.
\begin{lemma}\label{lem:key-obs} The positive and negative parts of
  $U_h$,
$$
U_h^\pm:=\max\{\pm U_h, 0\},
$$
satisfy
$$
(\Delta-\partial_t)(U_h^\pm)\geq - F_h^-\quad\text{in }Q_{3/4}.
$$
\end{lemma}
\begin{proof} It is clear that the inequality is satisfied in
  $Q_{3/4}^\pm$, so we will need to show the inequality near
  $(x_0,t_0)\in Q_{3/4}'$.  Suppose first that $U_h(x_0,t_0)>0$. Then,
  necessarily $u(x_0,t_0)>0$ and therefore
$$
(\Delta-\partial_t)u(x,t)=f(x,t)\quad\text{in }Q_\delta(x_0,t_0),
$$
for some small $\delta>0$. On the other hand,
$$
(\Delta-\partial_t)u(x,t-h)\leq f(x,t-h)\quad\text{in
}Q_\delta(x_0,t_0),
$$
in the sense of distributions, and taking the difference, we obtain
$$
(\Delta-\partial_t) U_h\geq F_h\quad\text{in }Q_\delta(x_0,t_0).
$$
We thus have
$$
(\Delta-\partial_t) U_h\geq -F_h^-\quad\text{in }\{U_h>0\}\cap
Q_{3/4}
$$
and a standard argument now implies that
$$
(\Delta-\partial_t)U_h^+\geq -F_h^-\quad\text{in
}Q_{3/4}.
$$
Indeed, for nonnegative $\eta\in C^\infty_0(Q_{3/4})$ and $\epsilon>0$
let
$$
\eta_\epsilon=\eta\,\chi(U_h/\epsilon),\quad\text{where }\chi\in
C^\infty(\R),\quad \chi\big|_{(-\infty,1]}=0,\quad
\chi\big|_{[2,\infty)}=1,\quad \chi'\geq 0.
$$
Since $U_h$ is continuous, $\eta_\epsilon$ is supported in $\{U_h>0\}$
and hence
$$
\iint_{Q_{3/4}}(\nabla U_h\nabla \eta_\epsilon+\partial_t U_h
\eta_\epsilon)\leq \iint_{Q_{3/4}} F_h^-\eta_\epsilon\leq
\iint_{Q_{3/4}} F_h^-\eta.
$$
On the other hand,
\begin{align*}
  \iint_{Q_{3/4}}\nabla U_h\nabla \eta_\epsilon&=\iint_{Q_{3/4}}(\nabla
                                                 U_h\nabla\eta)\chi(U_h/\epsilon)+\eta\frac1\epsilon\chi'(U_h/\epsilon)|\nabla
                                                 U_h|^2\\
                                               &\geq \iint_{Q_{3/4}}(\nabla
                                                 U_h\nabla\eta)\chi(U_h/\epsilon).
\end{align*}
Passing to the limit as $\epsilon\to 0+$, using the Dominated
Convergence Theorem, we then conclude
$$
\iint_{Q_{3/4}} (\nabla U_h\nabla\eta+\partial_t
U_h\eta)\chi_{\{U_h>0\}}\leq \iint_{Q_{3/4}} F_h^-\eta,
$$
which can be rewritten as
$$
\iint_{Q_{3/4}} (\nabla U_h^+\nabla \eta+\partial_t U_h^+\eta)\leq
\iint_{Q_{3/4}} F_h^-\eta.
$$
The proof for $U_h^-$ is similar.
\end{proof}

We will also need the following known estimate.
\begin{lemma}\label{lem:W22} Let $u$ be a weak solution of
  \textup{\eqref{eq:sig-u1}--\eqref{eq:sig-u2}}. Then
  $u\in W^{2,1}_2(Q_{\rho}^+)$ for any $\rho<1$ with
$$
\|D^2 u\|_{L_2(Q_\rho^+)}+\|\partial_t u\|_{L_2(Q_\rho^+)}\leq
C_{\rho,n} \left(\|u\|_{L_2(Q_1^+)}+\|f\|_{L_2(Q_1^+)}\right).
$$
\end{lemma}
The proof can be found in \cite[Lemma~6]{AU2}, and in the
Gaussian-weighted case in \cite{DGPT}.

Going back to the proof of Theorem~\ref{thm:vt-bdd}, we can now use
the interior $L_\infty$-$L_2$ estimates for subsolutions (see
\cite[Theorem~6.17]{Lie}) to write
$$
\|U_h^\pm\|_{L_\infty(Q_{1/2})}\leq
C_n\left(\|U_h\|_{L_2(Q_{3/4})}+\|F_h^-\|_{L_\infty(Q_{3/4})}\right).
$$
On the other hand, since
$$
U_h(x,t)=\frac{1}{h}\int_{t-h}^t \partial_t u(x,s)ds,
$$
we obtain that
\begin{align*}
  \|U_h\|_{L_2(Q_{3/4})}&\leq 2 \|U_h\|_{L_2(Q_{3/4}^+)}\leq 2\|\partial_t u\|_{L_2(Q_{5/6}^+)}\\
                        &\leq C_n\left(\|u\|_{L_2(Q_1^+)}+\|f\|_{L_2(Q_1^+)}\right),
\end{align*}
where in the last inequality we have applied Lemma~\ref{lem:W22}. It
is also clear that
$$
\|F_h\|_{L_\infty(Q_{3/4})}\leq \|\partial_t f\|_{L_\infty(Q_1^+)}.
$$
Letting now $h\to 0$, we then obtain the estimate
$$
\|\partial_t u\|_{L_\infty(Q_{1/2})}\leq
C_n\left(\|u\|_{L_2(Q_1^+)}+\|f\|_{L_2(Q_1^+)}+ \|\partial_t
  f\|_{L_\infty(Q_1^+)}\right),
$$
which readily implies the statement of Theorem~\ref{thm:vt-bdd}.
\end{proof}

\section{H\"older continuity of the time derivative at regular points}
\label{sec:hold-cont-time}

In formulation \eqref{eq:sig-u1}--\eqref{eq:sig-u2}, the free boundary
is given by
$$
\Gamma(u)=\partial_{Q_1'}\{(x',t)\in Q_1': u(x',0,t)>0\}.
$$
As shown in \cite{DGPT}, a successful study of the properties of the
free boundary near $(x_0,t_0)\in\Gamma(u)$ can be made by considering
the rescalings
$$
u_r(x,t)=u_r^{(x_0,t_0)}(x,t):=\frac{u(x_0+rx,t_0+r^2t)}{(H_u^{(x_0,t_0)}(r))^{1/2}},
$$
for $r>0$, and then studying the limits of $u_r$ as $r=r_j\to 0+$
(so-called \emph{blowups}). Here
$$
H_u^{(x_0,t_0)}(r):=\frac1{r^2}\int_{t_0-r^2}^{t_0}\int_{\R^n}
u^2(x,t)\psi^2(x) G(x_0-x,t_0-t) dxdt,
$$
where $\psi$ is a cutoff function, which is supported in $B_1$ and
equals $1$ in a neighborhood of $x_0$, and
$$
G(x,t)=\begin{cases}\frac{1}{(4\pi t)^{n/2}} e^{-|x|^2/4t}, &t>0,\\
  0, &t\leq 0
\end{cases}
$$
is the heat kernel.  Then a free boundary point
$(x_0,t_0)\in \Gamma(u)$ is called \emph{regular} if $u_r$ converges
in the appropriate sense to
$$
u_0(x,t)=c_n\Re(x_{n-1}+i |x_{n}|)^{3/2},
$$
as $r=r_j\to 0+$, after a possible rotation of coordinate axes in
$\R^{n-1}$. Note that this does not depend on the choice of the cutoff
function $\psi$ above. See \cite{DGPT} for more details and for a
finer classification of free boundary points based on a generalization
of Almgren's and Poon's frequency formulas.

Thus, let $\cR(u)$ be the set of regular free boundary points of $u$,
also known as the \emph{regular set} of the solution $u$. The
following result has been proved in \cite{DGPT}.

\begin{proposition}\label{prop:signor-known} Let $u\in
  H^{3/2,3/4}(Q_1^+\cup Q_1')$
  be a solution of the parabolic Signorini problem
  \textup{\eqref{eq:sig-u1}--\eqref{eq:sig-u2}} with
  $f\in H^{1,1/2}(Q_1^+\cup Q_1')$. Then the regular set $\cR(u)$ is a
  relatively open subset of $\Gamma(u)$. Moreover, if
  $(x_0,t_0)\in \cR(u)$, then there exists $\rho=\rho_u(x_0,t_0)>0$
  and a parabolically Lipschitz function $g:Q_\rho''(x_0'',t_0)\to\R$
  such that
  \begin{align*}
    \Gamma(u)\cap Q_\rho'(x_0,t_0)=\cR(u)\cap
    Q_\rho'(x_0,t_0)&=\{x_{n-1}=g(x'',t),x_n=0\}\cap Q_\rho'(x_0,t_0),\\
    \Lambda(u)\cap Q_\rho'(x_0,t_0)&=\{x_{n-1}\leq g(x'',t),x_n=0\}\cap Q_\rho'(x_0,t_0).
  \end{align*}
\end{proposition}
The parabolic Lipschitz continuity of the function $g$ above means
that for some constant $L$ (parabolic Lipschitz constant)
$$
|g(x'',t)-g(y'',s)|\leq L (|x''-y''|^2+|t-s|)^{1/2},\quad
(x'',t),(y'',s)\in Q_\rho''(x_0'',t_0).
$$

\medskip We are now ready to prove the following more precise version
of Theorem~\ref{thm:vt-cont-reg}.
\begin{theorem}\label{thm:ut-Holder-local} Let $u\in
  H^{3/2,3/4}(Q_1^+\cup Q_1')$
  be a solution of the parabolic Signorini problem
  \textup{\eqref{eq:sig-u1}--\eqref{eq:sig-u2}} with
  $f\in H^{2,1}(Q_1^+\cup Q_1)$, extended by the even symmetry in
  $x_n$ to $Q_1$. Then for any $(x_0,t_0)\in\cR(u)\cap Q_{1/4}'$ we
  have
$$
|\partial_t u(x,t)|\leq C(|x-x_0|^2+|t-t_0|)^{\alpha/2},\quad (x,t)\in
Q_{1/2}\setminus \Lambda(u),
$$
for some $\alpha=\alpha_u(x_0,t_0)>0$ and $C=C_u(x_0,t_0)$.
\end{theorem}
\begin{proof} Let $\rho=\rho_u(x_0,t_0)>0$ be as in
  Proposition~\ref{prop:signor-known}. Without loss of generality we
  may assume $\rho\leq 1/4$. Consider then the incremental quotients
  $U_h$ and $F_h$ defined in the proof of Theorem~\ref{thm:vt-bdd}. In
  addition to Lemma~\ref{lem:key-obs}, we then also have that
$$
U_h=0\quad\text{on }\Lambda_h,
$$
where
$$
\Lambda_h=\{(x',t):x_{n-1}\leq g(x'',t)-L h^{1/2},x_n=0\}\cap
Q_\rho'(x_0,t_0).
$$
Here $g$ is the function in the representation of
$\Lambda(u)\cap Q_\rho'(x_0,t_0)$ and $L$ is the parabolic Lipschitz
constant of $g$.  Then $\Lambda_h$ is a subgraph of a parabolically
Lipschitz function in $Q_\rho'(x_0,t_0) $, with the same parabolic
Lipschitz constant $L$ as $g$ (actually, just a shift of
$g$). Besides, from the assumption $(x_0,t_0)\in \Gamma(u)$, we have
that
$$(x_h,t_0):=(x_0-Lh^{1/2}e_{n-1},t_0)\in \Lambda_h.$$
Then, at every scale, $\Lambda_h$ has a positive thermal capacity at
$(x_h,t_0)$ (see e.g. \cite[\S3.2]{PS}). We then claim that
\begin{equation}\label{eq:Uh-est}
  U_h^\pm(x,t) \leq C(|x-x_h|^2+|t-t_0|)^{\alpha/2},\quad (x,t)\in Q_{1/2},
\end{equation}
with $\alpha>0$ depending only on the parabolic Lipschitz norm of $g$,
and $C$ depending only on $n$, $\rho$, and the $L_\infty$ norms of
$U_h$ and $F_h$. Since the latter are uniformly bounded by
$\|u\|_{L_2(Q_1)}$ and $\|f\|_{H^{2,1}(Q_1)}$, we can pass to the
limit as $h\to 0+$ to obtain
$$
|\partial_t u(x,t)|\leq C(|x-x_0|^2+|t-t_0|)^{\alpha/2},\quad (x,t)\in
Q_{1/2}\setminus \Lambda(u).
$$
Thus, to finish the proof, we need to
establish~\eqref{eq:Uh-est}. This, in principle, follows from
\cite[Theorem~6.32]{Lie}, but with the uniform density condition on
the complement (condition (A)) replaced with the uniform thermal
capacity condition that we have for $\Lambda_h$. Nevertheless, we give
a more direct proof below.

Fix $0<R<\rho$ and let $W$ solve the Dirichlet problem (see
Fig.~\ref{fig:Omegah})
\begin{align*}
  (\Delta -\partial_t)W=0&\quad\text{in
                           }\Omega_h(R):=[B_{R}(x_h)\times(t_0-R^2,t_0+R^2)]\addtocounter{footnote}{1}\footnotemark\cap [Q_1\setminus\Lambda_h],\\
  W=U_h^\pm&\quad\text{on }\partial_p \Omega_h(R).
\end{align*}%
\footnotetext{Note that $B_{R}(x_h)\times(t_0-R^2,t_0+R^2)$ is the
  ``full'' parabolic cylinder at $(x_h,t_0)$, while $Q_R(x_h,t_0)$ is
  the backward-in-time cylinder $B_{R}(x_h)\times(t_0-R^2,t_0]$}%
\definecolor{lightpink}{RGB}{230,211,211}%
\begin{figure}[t]
  \centering
  \begin{picture}(150,150)
    \put(4,0){\includegraphics[height=150pt]{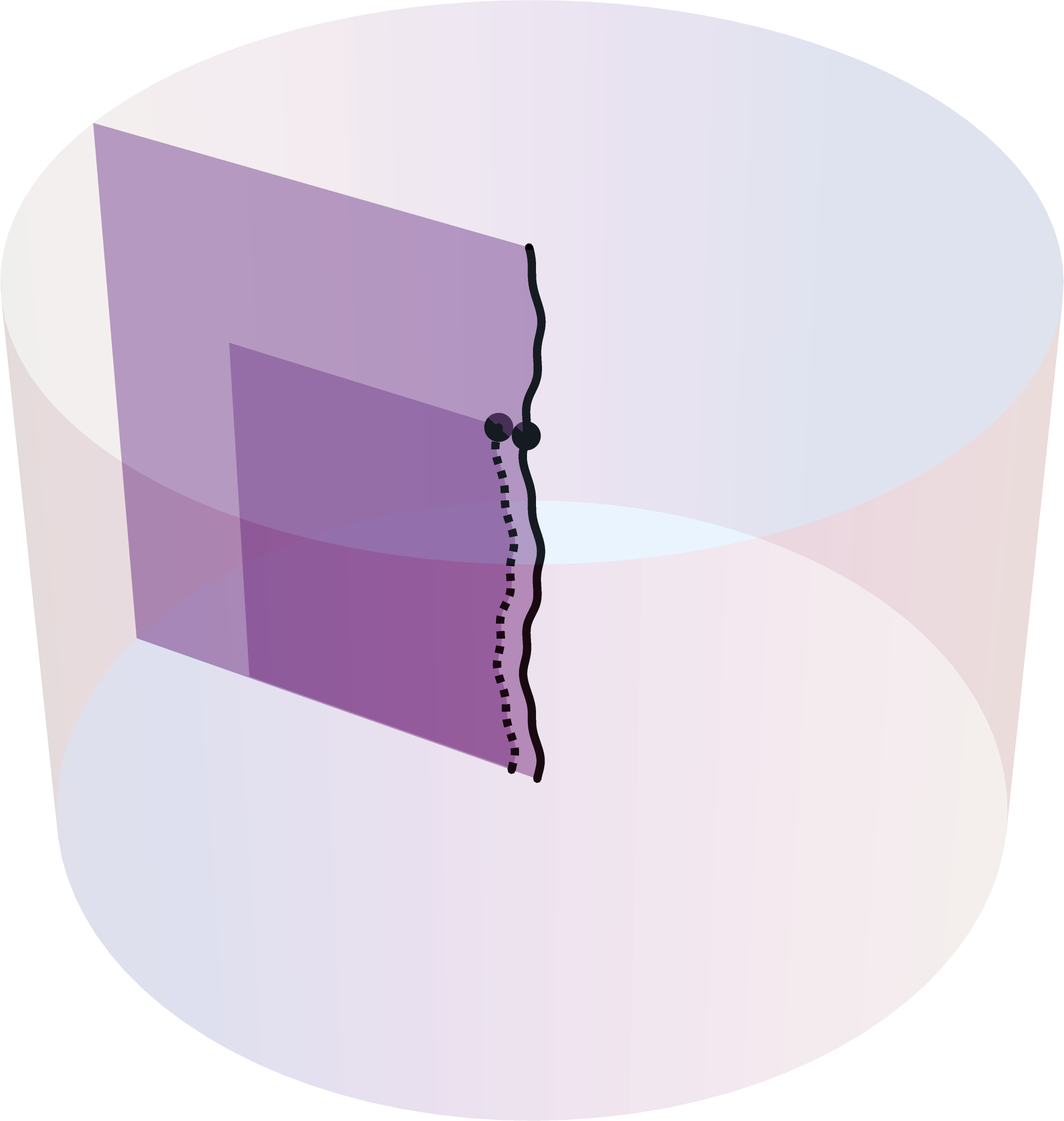}}
    \put(80,90){\small $(x_0,t_0)$} \put(47,100){\small $(x_h,t_0)$}
    \put(50,75){\small\color{lightpink} $\Lambda_h$}
  \end{picture}
  \caption{$\Omega_h(R)$}
  \label{fig:Omegah}
\end{figure}%
By using Lemma~\ref{lem:key-obs} and comparing $W$ with
$U_h^\pm+ C(|x-x_h|^2-(t-t_0)-2R^2)$ in $\Omega_h(R)$ with
$(2n-1)C\geq\|f\|_{H^{2,1}(Q_1)}\geq \|F_h^-\|_{L_\infty(Q_1)}$, we
see that
$$
U_h^\pm\leq W+CR^2\quad\text{on }\Omega_h(R).
$$
On the other hand, using a comparison with a barrier function as in
\cite[Lemma~3.2]{PS}, we have
$$
W(x,t)\leq
C\left(\frac{|x-x_h|^2+|t-t_0|}{R^2}\right)^{\beta/2}\sup_{\Omega_h(R)}
U_h^\pm,
$$
with $C$ depending only on the parabolic Lipschitz constant $L$ of $g$
and the dimension $n$.  Here we have used that
$\sup_{\Omega_h(R)}W=\sup_{\Omega_h(R)} U_h^\pm$, by the maximum
principle.  Denoting
$$
\omega(r)=\sup_{\Omega_h(r)} U_h^\pm,
$$
we then obtain
$$
\omega(r)\leq C\left(\frac{r}R\right)^\beta \omega(R)+CR^2,\quad
0<r\leq R.
$$
Choosing $0<\tau<1$ small so that $\theta=C\tau^\beta<1$, we then have
$$
\omega(\tau R)\leq \theta\, \omega(R)+CR^2.
$$
Then, a standard iterative argument (see \cite[Lemma~8.23]{GT}) gives
$$
\omega(R)\leq CR^\alpha, \quad R\leq \rho,
$$
for $\alpha>0$, which establishes \eqref{eq:Uh-est} and completes the
proof of the theorem.
\end{proof}

\begin{corollary}\label{cor:C1a-xt} Let $u$, $(x_0,t_0)$, $\rho$ and $g$ be as in
  Proposition~\ref{prop:signor-known}. Then
  $\Gamma(u)\cap Q_\rho'(x_0,t_0)$ is an $(n-2)$-dimensional
  $C^{1,\alpha}$ surface both in the $x$ and $t$ variables.
\end{corollary}
\begin{proof} One argues precisely as in the proof of
  \cite[Theorem~11.6]{DGPT} to show that
$$
\frac{\partial_{x_j} u}{\partial_{x_{n-1}}u},\ j=1,\ldots,{n-2},\quad
\frac{\partial_t u}{\partial_{x_{n-1}}u}\in
H^{\alpha,\alpha/2}(Q_{\rho/2}'(x_0,t_0)),
$$
by the boundary Harnack principle in parabolic slit domains
\cite[\S7]{PS}. The argument works for $\partial_t u$ since we now
know that it continuously vanishes on $\Lambda(u)\cap Q_\rho(x_0,t_0)$
by Theorem~\ref{thm:ut-Holder-local}. Consequently, the level sets
$\{u=\epsilon\}\cap Q_{\rho/2}' (x_0,t_0)$ are given as graphs
$$
x_{n-1}=g_\epsilon(x'',t)
$$
with uniform estimates on the H\"older norms of
$\partial_{x_j}g_\epsilon$, $j=1,\ldots,n-2$, and
$\partial_t g_\epsilon$. This then implies the H\"older continuity of
$\partial_{x_j}g$ and $\partial_t g$ and completes the proof of the
corollary.
\end{proof}

\begin{remark} Very recently, in \cite{BSZ}, it was proved that when
  the thin obstacle $\phi$ is identically zero, the free boundary is
  $C^\infty$ both in the $x$ and $t$ variables near regular free
  boundary points. More precisely, the function $g$ in the
  representation of $\Gamma(u)$ in Proposition~\ref{prop:signor-known}
  and Corollary~\ref{cor:C1a-xt} is $C^\infty$. This is established by
  extending the higher-order boundary Harnack principle in \cite{DS}
  to parabolic slit domains, and using an argument similar to the
  proof of Corollary~\ref{cor:C1a-xt} above. An important ingredient
  in the proof is our Theorem~\ref{thm:vt-cont-reg}, which allows the
  iteration steps in the $t$ variable.
\end{remark}

\end{document}